\documentclass[reqno,a4paper,12pt]{amsart}
\usepackage{amsmath,amscd,amsfonts,amssymb}
\usepackage{mathrsfs,dsfont}

\numberwithin{equation}{section}
\numberwithin{figure}{section}

\newcommand\R{\mathbb{R}}

\newcommand\Z{\mathbb{Z}}

\newcommand\lam{\lambda}
\newcommand\Lam{\Lambda}

\newcommand\Om{\Omega}

\newcommand\1{\mathds{1}}
\newcommand\eps{\varepsilon}

\renewcommand\le{\leqslant}
\renewcommand\ge{\geqslant}
\renewcommand\leq{\leqslant}
\renewcommand\geq{\geqslant}
\newcommand\sbt{\subset}

\newcommand{\ft}[1]{\widehat #1}
\newcommand{\dotprod}[2]{\langle #1 , #2 \rangle}

\newcommand{\supp}{\operatorname{supp}}

\newcommand{\cm}{\complement}

\theoremstyle{plain}
\newtheorem{thm}{Theorem}[section]

\newtheorem*{claim*}{Claim}

\newcommand{\thmref}[1]{Theorem~\ref{#1}}
\newcommand{\secref}[1]{Section~\ref{#1}}

\newcommand\del{\delta}

\theoremstyle{definition}
\newtheorem{definition}[thm]{Definition}

\newtheorem*{definition*}{Definition}
\newtheorem*{remarks*}{Remarks}
\newtheorem*{remark*}{Remark}
\newtheorem{remark}[thm]{Remark}

\newenvironment{enumerate-roman}
{\begin{enumerate}
\addtolength{\itemsep}{5pt}
}
{\end{enumerate}}

\newenvironment{enumerate-alph}
{\begin{enumerate}
\addtolength{\itemsep}{5pt}
}
{\end{enumerate}}

\newenvironment{enumerate-num}
{\begin{enumerate}
\addtolength{\itemsep}{5pt}
}
{\end{enumerate}}

\newenvironment{enumerate-text}
{\begin{enumerate}
\addtolength{\itemsep}{5pt}
}
{\end{enumerate}}

\begin{document}

\title{Geometric implications of weak tiling}

\author[M. Kolountzakis]{Mihail N. Kolountzakis}
\address{Department of Mathematics and Applied Mathematics, University of Crete, Voutes Campus, 70013 Heraklion, Greece and Institute of Computer Science, Foundation of Research and Technology Hellas, N. Plastira 100, Vassilika Vouton, 700 13, Heraklion, Greece}
\email{kolount@gmail.com}

\author[N. Lev]{Nir Lev}
\address{Department of Mathematics, Bar-Ilan University, Ramat-Gan 5290002, Israel}
\email{levnir@math.biu.ac.il}

\author[M. Matolcsi]{M\'at\'e Matolcsi}
\address{HUN-REN Alfr\'ed R\'enyi Institute of Mathematics, Re\'altanoda utca 13-15, H-1053, Budapest, Hungary and Department of Analysis and Operations Research, Institute of Mathematics, Budapest University of Technology and Economics, M\H uegyetem rkp. 3., H-1111 Budapest, Hungary}
\email{matomate@renyi.hu}

\dedicatory{Dedicated to the 100th anniversary of Bent Fuglede's birthday}
\date{September 12, 2025}
\thanks{N.L.\ is supported by ISF Grant No.\ 1044/21.
M.M.\ is supported by the Hungarian National Foundation for Scientific Research, Grants No. K146387 and KKP 133819.}

\begin{abstract}
The notion of weak tiling played a key role in the proof of Fuglede's spectral set conjecture for convex domains, due to the fact that every spectral set must weakly tile its complement. In this paper, we revisit the notion of weak tiling and establish some geometric properties of sets that weakly tile their complement. If $A \subset \mathbb{R}^d$ is a convex polytope, we give a direct and self-contained proof that $A$ must be symmetric and have symmetric facets. If $A \subset \mathbb{R}$ is a finite union of intervals, we give a necessary condition on the lengths of the gaps between the intervals.
\end{abstract}

\maketitle

\section{Introduction}

A bounded, measurable set $A \sbt \R^d$ is called
\emph{spectral} if it admits an orthogonal basis
consisting of  exponential functions.
Fuglede  famously conjectured \cite{Fug74} that $A$ is a spectral set
if and only if it can tile the space by translations.
This conjecture inspired extensive research over the years,
see \cite{Kol24} for the history of the
problem and an overview of the known (positive as well as negative)
related results.

A major recent result  states that the Fuglede conjecture
holds for convex domains in all dimensions  \cite{LM22}.
 A key role in the proof is played by the concept of weak tiling,
 introduced in the same paper as a relaxation of proper tiling.

\begin{definition}
We say that a bounded, measurable set $A \sbt \R^d$
\emph{weakly tiles its complement}  $A^\cm = \R^d \setminus A$,
if there  exists a positive, locally finite measure $\nu$
on $\R^d$ such that $\1_{A} \ast \nu = \1_{A^\cm}$ a.e.
In this case, $\nu$ is called a \emph{weak tiling measure}
for the set $A$.
\end{definition}

If the measure $\nu$ is a sum of unit masses, then
the weak tiling becomes a proper tiling of the complement $A^\cm$
by translated copies of $A$.

The role of weak tiling in the theory of spectral sets is
due to the fact that every spectral
set must weakly tile its complement, see \cite[Theorem 1.5]{LM22}.
This result can be viewed as a weak form of the
``spectral implies tile'' part of  Fuglede's conjecture. 
Note that there exist spectral sets which do not tile the space 
properly, as first constructed in \cite{Tao04}.

In the present paper we revisit the notion of weak tiling,
and establish some geometric properties of sets that weakly tile their complement.
In particular, these include all  the spectral sets, as well as the
sets which can tile properly by translations.

In \secref{sec2} we consider the case where $A$ is a convex polytope, and give a new self-contained proof of the fact that $A$ must be centrally symmetric and have centrally symmetric facets. In \secref{sec3} we consider the case where $A\subset \R$ is a finite union of intervals, and give a necessary condition on the lengths of the gaps appearing between the intervals. Finally, in \secref{sec4} we pose some open problems.


\section{Weak tiling by convex polytopes}
\label{sec2}

\subsection{}
In this section  we focus on  the case where
$A \sbt \R^d$ is a convex body, that is,
a compact, convex set with nonempty interior.
Note that  while in general, a set that weakly tiles
its complement need not tile properly 
(as an example, take any spectral set which does not tile),
it was shown in \cite{KLM23} that for the
class of convex bodies weak tiling implies tiling.

\begin{thm}[{\cite[Theorem 1.4]{KLM23}}]
\label{thmCVXBWT}
    Let $A$ be a convex body in $\R^d$, and assume that
    $A$ weakly tiles its complement. Then $A$ is a convex
 polytope which can also tile $\R^d$ properly by translations.
\end{thm}

The proof of this result consists of several ingredients. First,
due to \cite{Ven54}, \cite{McM80}, in order
to prove that a convex body $A \sbt \R^d$ tiles  by translations,
 it suffices to show that $A$ satisfies the following four conditions:
\begin{enumerate-num}
\item \label{vm:i} $A$ is a convex polytope;
\item \label{vm:ii} $A$ is centrally symmetric;
\item \label{vm:iii} all the facets of $A$ are centrally symmetric;
\item \label{vm:iv} each belt of $A$ consists of either $4$ or $6$ facets.
\end{enumerate-num}
We recall that a \emph{facet} of a convex polytope $A \sbt \R^d$
 is a face of dimension $d-1$.

It was proved in \cite[Theorem 4.1]{LM22} that
if $A$ is a  convex body in $\R^d$, then
the weak tiling assumption implies condition \ref{vm:i}. It was also proved
\cite[Theorem 6.1]{LM22} that the
weak tiling assumption together with \ref{vm:i}, \ref{vm:ii}
and \ref{vm:iii} implies  \ref{vm:iv}.

Conditions  \ref{vm:ii} and  \ref{vm:iii} were
shown in \cite{Kol00}, \cite{KP02}, \cite{GL17}
to follow from the assumption
that the convex polytope $A$ is spectral, but these proofs
do not work if we only
assume that $A$ weakly tiles its complement.

The fact that
the weak tiling assumption together with \ref{vm:i}
implies both  \ref{vm:ii} and  \ref{vm:iii},
was proved in \cite{KLM23}, based on an elaborate machinery developed in \cite{LL21}. In fact, conditions \ref{vm:ii} and  \ref{vm:iii} were obtained as a consequence of \cite[Theorem 1.3]{KLM23},
which states that if a (convex or non-convex) polytope
$A \sbt \R^d$ weakly tiles its complement, then $A$ is equidecomposable
by translations to a cube of the same volume.

Our goal in the present section is to give a more direct and self-contained proof of
the fact that
the weak tiling condition together with \ref{vm:i}
implies both  \ref{vm:ii} and  \ref{vm:iii}.

\begin{thm}
\label{thmB1}
Let $A$ be a convex polytope in $\R^d$. If
$A$   weakly tiles its complement, then $A$ must be centrally symmetric and have
centrally symmetric facets.
\end{thm}

The proof requires some auxiliary results, which
we will introduce first.  Then, we will proceed with
the proof of \thmref{thmB1}.

\subsection{Zero-free regions of the Fourier transform}
If $f \in L^1(\R^d)$ then
its Fourier transform  is defined by
\begin{equation}
\label{eq:Z1.1}
\ft f (\xi)=\int_{\R^d} f (x) \, e^{-2\pi i\langle \xi,x\rangle} dx, \quad \xi \in \R^d.
\end{equation}

Given a vector $v \in \R^d$,
a positive integer $R$,  and $\eps > 0$,  we define the set
\begin{equation}
\label{eqSdef}
    S = S(v, R, \eps) =
    \{ nv + w : n \in \Z, \; |n| \notin \{1,2,\dots,R-1\}, \; w \in \R^d, \; |w|< \eps\}.
\end{equation}
The set $S$ is thus a union of open balls
of radius $\eps$ centered at integral multiples
of the vector $v$, but only those multiples $nv$
for which either $n=0$ or $|n| \ge R$.
We observe
that $S(v, R, \eps)$ is contained in a cylinder of width $\eps$ along the vector $v$.

\begin{thm}
\label{thmFA}
Let $A$ be a convex polytope in $\R^d$. Assume
that $A$ is not centrally symmetric, or that at least
one of the facets of $A$ is not centrally symmetric.
Then there exist
a nonzero vector $v \in \R^d$, a positive integer $R$
and $\eps > 0$, such that
the Fourier transform $\ft{\1}_A$
has no zeros in the set $S =  S(v, R, \eps)$.
\end{thm}

\begin{proof}
This result was essentially established
in \cite[Sections 3 and 4]{GL17}. The proof is based on
Minkowski's theorem, which states that a
convex polytope $A$ is centrally symmetric
if and only if each facet $F$ of $A$
has a parallel facet $F'$ of $A$ such that $|F|=|F'|$
(see, for instance, \cite[Corollary 18.1]{Gru07}).

If $A$ is not centrally symmetric, then
the   conclusion of \thmref{thmFA}
follows from standard estimates, see
 equation (3.5) in \cite[Section 3]{GL17},
 where an even larger zero-free region is given,
which is obtained from a cylinder of width $\eps$
 by the removal of a large ball around the origin.

The case where $A$ is centrally symmetric
but has a non-centrally symmetric facet, is more intricate.
In this case, the conclusion of \thmref{thmFA}
is obtained as a combination of equations (4.6), (4.15) and (4.16) in
\cite[Section 4]{GL17}.

It remains to notice that $\ft{\1}_A$ also has no zeros
in some neighborhood of the origin, which is obvious
as $\ft{\1}_A$ is a continuous function with
 $\ft{\1}_A(0) = m(A) > 0$.
\end{proof}

\subsection{Translation-bounded measures}

A (complex) measure $\mu$ on $\R^d$  is said to be \emph{translation-bounded}
 if for every (or equivalently, for some) open ball $B$ we have
\begin{equation}
\sup_{x \in \R^d} |\mu|(B+x) < +\infty.
\end{equation}
If  $\mu$ is a translation-bounded measure on $\R^d$,
then $\mu$ is a tempered distribution.
Hence, the Fourier transform $\ft\mu$
is again a tempered distribution, defined by
$\ft\mu(\varphi)=\mu(\ft{\varphi})$
for any Schwartz function $\varphi$. 

If $f$ is a function in $L^1(\R^d)$ and
$\mu$ is a translation-bounded measure on $\R^d$,
then the convolution $f \ast \mu$ is
a translation-bounded measure which is also
a locally integrable function on $\R^d$,
which can be
defined (uniquely up to equality a.e.)
by the condition that
$(f \ast \mu) \ast \varphi =
f \ast (\mu \ast \varphi)$
for every continuous, compactly supported
function $\varphi$.

If $f \in L^1(\R^d)$ then we denote by
\begin{equation}
Z(f) := \{ \xi \in \R^d : \ft{f}(\xi)=0\}
\end{equation}
the (closed) set of   zeros of
the Fourier transform   $\ft{f}$.

\begin{thm}
\label{thmC1}
Let $f\in L^1(\R^d)$, $\int f \neq 0$,
and let $\mu$ be
a translation-bounded measure on $\R^d$.
If we have $f \ast \mu = 1$ a.e., then
$\ft{\mu} = (\int f)^{-1} \cdot \delta_0$
in the open set
$Z(f)^\complement$.
\end{thm}

\begin{proof}

This result can be proved in a similar way to
\cite[Theorem 4.1]{KL16}. The details are as follows.
First we observe that it
 suffices to prove the following claim:
Let $B$ be an open ball contained in
$Z(f)^\cm$, and let  $\psi$ be a smooth function
whose closed
support is contained in $B$. Then
$\ft\mu( \psi ) = (\int f)^{-1} \psi(0)$.

To prove this, let
 $K \sbt B$ be a closed ball containing
$\supp( \psi)$. Since $\ft f$ does not vanish on $K$, then
due to Wiener's theorem
there exists  $g\in L^1(\R^d)$ such that $\ft{f}\cdot \ft{g}=1$ on $K$
(see \cite[pp.\ 150--152]{Hel10}
for a proof in the one-dimensional case; the
multi-dimensional case is similar).
There is a Schwartz  function $\varphi$
such that $\ft\varphi=\psi$, so this function $\varphi$  satisfies
$\ft{\varphi}\cdot \ft{g} \cdot \ft{f}=\ft{\varphi}$.
This implies that
$\varphi\ast g\ast f=\varphi$, hence
\begin{align}
\label{eqTB.1}
  \ft\mu(\psi)  &= \mu(\ft \psi)
= \int \varphi(-t)d\mu(t)
=\int (\varphi \ast g\ast f )(-t) d\mu(t)\\
&=\int \Big( \int ( \varphi\ast g)(-x) f(x-t)dx \Big) d\mu(t).
\end{align}
If we can exchange the order of integrals, we obtain
\begin{equation}
\label{eqTB.2}
  \ft\mu(\psi) =  \int
  (\varphi\ast g) (-x) \Big( \int f(x-t) d\mu(t) \Big) dx= \int ( \varphi\ast g) (-x) dx,
\end{equation}
because the inner integral is $(f \ast \mu ) (x) $,
which is $1$ a.e.\ by assumption.
Furthermore,
\begin{equation}
\label{eqTB.3}
\int (\varphi\ast g)(-x) dx=\ft\varphi(0)\ft g(0)= \ft f(0)^{-1} \ft\varphi(0)
= ({\textstyle \int f})^{-1} \psi(0),
\end{equation}
where the second equality holds since we have
$\ft{\varphi}\cdot \ft{g} \cdot \ft{f}=\ft{\varphi}$,
and $\ft f(0) = \int f \neq 0$.

It thus remains to justify the exchange of integrals.
Indeed, observe that
\begin{align}
\label{eqTB.5}
&\int \Big( \int ( |\varphi|\ast |g| ) (-x) \cdot |f(x-t)|dx \Big) |d\mu(t)| \\
& \qquad\qquad =\int (|f|\ast |g|)(-x)  \Big( \int |\varphi(x-t)| \, |d\mu(t)| \Big) dx,
\end{align}
and the latter integral is finite,
because $|f|\ast |g|\in L^1(\R^d)$, and the inner integral is a bounded function of $x$ since
$\varphi$ is a Schwartz function and $\mu$ is a translation-bounded measure.
This justifies the exchange of integrals and completes the proof.
\end{proof}

\subsection{}
After these preliminary results we are ready to prove the main theorem.

\begin{proof}[Proof of \thmref{thmB1}]
Let $A$ be a convex polytope in $\R^d$,
and assume that $A$ weakly
 tiles its complement, so there is a positive
measure $\nu$ such that
$\1_A \ast \nu = \1_{A^\complement}$ a.e.
According to \cite[Lemma 2.4]{LM22} the measure $\nu$
must be  translation-bounded. Hence the
measure $\mu := \delta_0 + \nu$ is also translation-bounded,
and satisfies $\1_A \ast \mu = 1$ a.e.
In turn, \thmref{thmC1}  implies that
$\ft{\mu} = m(A)^{-1} \cdot \delta_0$
in the open set $Z(\1_A)^\cm$.

We must prove that
$A$ is centrally symmetric and has
centrally symmetric facets.
Suppose to the contrary that this is not the case.
Then by \thmref{thmFA}
there exist
a nonzero vector $v \in \R^d$, a positive integer $R$
and $\eps > 0$, such that $\ft{\1}_A$
has no zeros in the set $S =  S(v, R, \eps)$.
It follows  that
$\ft{\mu} = m(A)^{-1} \cdot \delta_0$
in the open set $S$.

Now suppose that we are given a real-valued
Schwartz function
$g$ on $\R^d$ satisfying
\begin{equation}
\label{eqGC1}
    \supp(g) \sbt S,
    \quad
    \ft{g} \geq 0.
\end{equation}
Then we have
\begin{equation}
\label{eqGineq}
    \int_{\R^d} g(x) dx = \ft{g}(0) \leq
    \ft{g}(0) + \int \ft{g}(\xi) d \nu(\xi) =
     \int \ft{g}(\xi) d \mu(\xi),
\end{equation}
where the inequality in \eqref{eqGineq} is due to $\ft{g}$
being a nonnegative function and $\nu$
being a positive measure. On the  other hand,
we have
\begin{equation}
     \int \ft{g}(\xi) d \mu(\xi)= \mu(\ft{g}) =
     \ft{\mu}(g) = m(A)^{-1} g(0),
\end{equation}
where the last equality holds since we have
$\supp(g) \sbt S$ and
$\ft{\mu} = m(A)^{-1} \cdot \delta_0$
in the open set $S$. We conclude that
\begin{equation}
\label{eqGC2}
\int_{\R^d} g(x) dx \leq m(A)^{-1} g(0)
\end{equation}
for every real-valued Schwartz function $g$
satisfying \eqref{eqGC1}.
We will show that this leads to a contradiction,
by constructing an example of a
real-valued Schwartz function $g$ satisfying
 \eqref{eqGC1}, but such that
 \eqref{eqGC2} does not hold.

We choose a nonnegative Schwartz function $\varphi$
such that $\int \varphi =1$,
$\varphi$ is supported in
 the open ball of radius
$\eps$ centered at the origin,
and $\ft{\varphi} \geq 0$. We
 define also the trigonometric polynomial
$p_N(t) := K_N(R \, t)$,
where $K_N$ is the classical Fej\'{e}r kernel,
\begin{equation}
\label{eqFej}
K_N(t) = \sum_{|n| < N} \Big(1 - \frac{|n|}{N} \Big)
e^{2 \pi i n t}, \quad t \in \R.
\end{equation}
Then $p_N$ is nonnegative, $p_N(0) = N$,
the Fourier coefficients $\ft{p}_N(n)$,
$n \in \Z$, are also nonnegative,
$\ft{p}_N(0)=1$, and we
have  $\ft{p}_N(n)=0$ if $0 < |n| < R$.

Finally, we define the function
\begin{equation}
\label{eqGN1}
g_N(x) := \sum_{n \in \Z}
\ft{p}_N(n) \, \varphi(x - n v),
\quad x \in \R^d.
\end{equation}
Notice that there are only finitely many
nonzero terms in the sum \eqref{eqGN1}, and that the nonzero
terms correspond to integers $n$ such that either $n=0$
or $|n| \ge R$.
Hence $g_N$ is a real-valued (in fact, nonnegative)
Schwartz function such that
$\supp(g_N)$ is contained in the set
$S =  S(v, R, \eps)$. The Fourier transform
of $g_N$ is given by
\begin{equation}
\label{eqGN2}
\ft{g}_N(\xi) = \ft{\varphi}(\xi) \sum_{n \in \Z}
\ft{p}_N(n) \, e^{-2 \pi i n \dotprod{v}{\xi}}
= \ft{\varphi}(\xi) \, p_N(- \dotprod{v}{\xi}),
\end{equation}
hence $\ft{g}_N$ is a nonnegative function. We
conclude that $g_N$ satisfies the conditions
\eqref{eqGC1}.

To complete the proof
we will show that if $N$
is sufficiently large, then $g_N$ does not
satisfy \eqref{eqGC2}. Indeed, we may assume
that $\eps < \frac1{2} |v|$ which ensures
that the terms in the sum
\eqref{eqGN1} have pairwise disjoint supports.
This implies that
\begin{equation}
\label{eqGN4}
g_N(0) = \varphi(0),
\end{equation}
so that the value $g_N(0)$ does not depend on $N$.
On the other hand, using \eqref{eqGN2} we have
\begin{equation}
\label{eqGC5}
\int_{\R^d} g_N(x) dx = \ft{g}_N(0)
= \ft{\varphi}(0) p_N(0) = N,
\end{equation}
which can be arbitrarily large,
contradicting \eqref{eqGC2}. We therefore
arrive at the desired contradiction, and
\thmref{thmB1} is thus proved.
\end{proof}


\section{Weak tiling by finite unions of intervals}
\label{sec3}

In dimensions $d=1$ and $2$, Fuglede's spectral set conjecture is still open in both directions. The important special case where $\Om \sbt \R$ is a finite union of intervals was recently studied in \cite{DDF25}, where several necessary conditions were established for $\Om$ to be a spectral set. One of these conditions states that the lengths of the gaps between the intervals must be representable as sums of lengths (with multiplicities) of some of the intervals composing $\Om$.
In this section, in \thmref{thmG2} below, we will prove this in more generality, assuming only that $\Om$ weakly tiles its complement.

We remark here that Fuglede's conjecture was proved for a union of \emph{two} intervals, see
\cite{Lab01}, but it is still open for a union of any higher number of intervals.

\subsection{Spectral unions of intervals}
Let us begin by considering the more restrictive situation
where $\Om$ does not only weakly tile its complement,
but moreover is assumed to be spectral.
In this case, the following result was obtained in
\cite{DDF25}.

\begin{thm}[{\cite[Theorem 3.1(i)]{DDF25}}]
\label{thmD5.4}
Let  $\Om = \bigcup_{i=1}^{n} (a_i, b_i)$,
$a_1 < b_1 < a_2 < b_2 < \dots < a_n < b_n$,
and assume that $\Om$ is a spectral set.
 Then each gap length $a_{k+1} - b_k$
is representable in the form
$ \sum_{i=1}^{n} p_i (b_i - a_i)$
where $p_i$ are nonnegative integers.
\end{thm}

Below we give an alternative proof of this result,
based on a simple argument combined
with existing results in the literature.

\begin{proof}[Proof of \thmref{thmD5.4}]
Since $\Om$ is a spectral set,
it admits a spectrum $\Lam$, that is,
there is a countable set of frequencies
$\Lam \sbt \R$
such that the system of exponential functions
$\{\exp 2 \pi i {\lam}{x} \}$, $\lam \in \Lam$,
forms an orthogonal basis in $L^2(\Om)$.
It is known that for a finite
union of intervals,
any spectrum $\Lam$ is periodic
\cite{BM11}, \cite{Kol12}, \cite{IK13}.
In this case,
the proof of \cite[Theorem 1.5]{LM22}
yields that  $\Om$ admits a weak tiling
measure $\nu$ with the additional property that
$\ft \nu$ is a periodic measure. This implies
that $\nu$ is a pure point measure which is
supported on some arithmetic progression.
In particular, $\supp(\nu)$ is a locally
finite set.

The weak tiling condition thus says that
$\1_{\Om^\cm} = \sum_t \nu(t) \1_{\Om + t}$ a.e., where
$t$ goes through the atoms of $\nu$.
In particular, if $t$ is an atom of $\nu$, then
$\Om + t \sbt \Om^\cm$.
If we now fix one of the gaps  $I = (a, b)$, then
$\1_{I} = \sum_t \nu(t) \1_{I \cap (\Om + t)}$ a.e.
We observe that in this sum  it  suffices that $t$ only runs
through a finite set of atoms,
since  $\supp(\nu)$ is a locally finite set and hence
$\Om + t $ does not intersect $I$ for all but finitely many
atoms $t$.

Note that for any atom $t$, the set $ I \cap (\Om + t) $
is a union of several of the intervals
composing $\Om$. Hence the
indicator function $\1_I$ is a
 finite linear combination with positive   coefficients
 of indicator functions of intervals $I_j$,
where the length of each $I_j$
belongs to the
finite set $L = \{b_1-a_1, \dots, b_n-a_n\}$.
Hence, if we denote $I_j = (x_j, y_j)$, then we have
$\1_{(a,b)} = \sum_{j=1}^{m} w_j \1_{(x_j, y_j)}$
a.e., where $w_j$ are strictly positive scalars.

This equality holds also in the sense of distributions,
hence differentiating yields
\begin{equation}
\label{eq:C2.1}
\del_{a} - \del_{b}
 = \sum_{j=1}^{m} w_j (\del_{x_j} - \del_{y_j}).
\end{equation}
The measure on the right hand side of  \eqref{eq:C2.1}
thus must assign zero mass to all the points $x_j$ and $y_j$
except for $a$ and $b$, while it assigns the mass $+1$ to $a$
and $-1$ to $b$.

 It follows from  \eqref{eq:C2.1} that at least one of the
 points $x_j$ must be equal to $a$; let's say it is the
 point $x_{j_1}$.
 Then we have $a = x_{j_1} < y_{j_1}$.
The $j_1$'th term on the right hand side of  \eqref{eq:C2.1}
gives a negative mass to $y_{j_1}$, hence if $y_{j_1} < b$
then there must exist $j_2$ such that $x_{j_2} = y_{j_1}$.
So we have $a = x_{j_1} < y_{j_1} =  x_{j_2} < y_{j_2}$.

Continuing this way, we obtain a sequence
\begin{equation}
\label{eq:C2.2}
a = x_{j_1} < y_{j_1} =  x_{j_2} < y_{j_2} = x_{j_3} < y_{j_3} = \dots
\end{equation}
which we can continue as long as $y_{j_s} < b$. Since there
are finitely many points $x_j, y_j$, it follows that
there must exist $N$ such that $y_{j_N} = b$. This implies that
\begin{equation}
\label{eq:C2.3}
b - a = \sum_{s=1}^{N} (y_{j_s} - x_{j_s}),
\end{equation}
so the gap length $b-a$ is a
sum of lengths of intervals
$I_{j_s} = (x_{j_s}, y_{j_s})$
and these
lengths belongs to the
finite set $L = \{b_1-a_1, \dots, b_n-a_n\}$.
This implies that the gap length
$b-a$  is  representable
in the form
$ \sum_{i=1}^{n} p_i (b_i - a_i)$
where $p_i$ are nonnegative integers.
\end{proof}

\subsection{Weak tiling measures are pure point}
We now  turn to the more general case
where $\Om \sbt \R$ is a finite union of intervals
that weakly tiles its complement, but is
not assumed to be spectral.
This case is more intricate, since the weak tiling measure
$\nu$
is not known a priori to be supported on a locally finite set,
nor even to be purely atomic.

We begin by showing
that any weak tiling measure must
be purely atomic.

\begin{thm}
\label{thmG1}
Let $\Om \sbt \R$ be a finite union of intervals,
and assume that $\nu$ is a weak tiling measure for $\Om$.
Then $\nu$ must be a pure point measure.
\end{thm}

\begin{proof}
We may assume that $\Om$ is a nonempty open set given by
$\Om = \bigcup_{j=1}^{n} (a_j, b_j)$,
where
$a_1 < b_1 < a_2 < b_2 < \dots < a_n < b_n$.
By our assumption, the measure $\mu  = \del_0 + \nu$
satisfies $\1_\Om \ast \mu = 1$ a.e.
It is straightforward to verify that if  $\varphi$
is a smooth, compactly supported  function on $\R$, then
$\varphi \ast  (\del_{a_j} - \del_{b_j})  =  \varphi' \ast  \1_{(a_j, b_j)}$
for each $j$. As a consequence,
\begin{equation}
\label{eq:C3.4.4}
\varphi \ast  \sum_{j=1}^{n} (\del_{a_j} - \del_{b_j}) \ast \mu
 =  \varphi' \ast  \1_{\Om} \ast \mu =
  \varphi' \ast  1 = 0.
\end{equation}
As this holds for every smooth, compactly
supported  function $\varphi$, we conclude that
\begin{equation}
\label{eq:C3.5.1}
\sum_{j=1}^{n} (\del_{a_j} - \del_{b_j}) \ast \mu = 0.
\end{equation}

The measure $\mu$ has a  unique decomposition
$\mu = \mu_d + \mu_c$ into a sum of
a pure point measure
$\mu_{d}$  and a continuous measure
 $\mu_c$. We observe that
 $\sum_{j=1}^{n} (\del_{a_j} - \del_{b_j}) \ast \mu_d$
is a pure point measure, while
 $\sum_{j=1}^{n} (\del_{a_j} - \del_{b_j}) \ast \mu_c$
 is a continuous measure. The sum of these two measures
vanishes by \eqref{eq:C3.5.1}, hence
 both measures must be zero. In particular,
\begin{equation}
\label{eq:C3.5.3}
\sum_{j=1}^{n} (\del_{a_j} - \del_{b_j}) \ast \mu_d = 0.
\end{equation}

Now, if $\varphi$ is again
 a smooth, compactly supported  function on $\R$, then as before
\begin{equation}
\label{eq:C3.5.5}
  \varphi' \ast  \1_{\Om} \ast \mu_d = \varphi \ast
\sum_{j=1}^{n} (\del_{a_j} - \del_{b_j}) \ast \mu_d = 0.
\end{equation}
If $\varphi_1$, $\varphi_2$ are two smooth,
 compactly supported  functions,
  $\int \varphi_1 = \int \varphi_2$,
 then there is a
  smooth, compactly supported  function
 $  \varphi$ such that $ \varphi' = \varphi_2 - \varphi_1  $.
 Hence  \eqref{eq:C3.5.5} yields that
$ \varphi_1 \ast  (\1_{\Om} \ast \mu_d ) = \varphi_2 \ast  (\1_{\Om} \ast \mu_d) $.
This is possible only if  the function
$\1_\Om \ast \mu_d$ is constant a.e.
But now recall that $\mu = \del_0 + \nu$, and therefore
\begin{equation}
\label{eq:C3.5.6}
\1_{\Om} = \1_{\Om} \ast \del_0  \le
\1_{\Om} \ast \mu_d  \le \1_{\Om} \ast \mu = 1 \quad \text{a.e.}
\end{equation}
As the function
$\1_\Om \ast \mu_d$ is  constant a.e., it follows that
its constant value must be $1$. In turn, we conclude that
$\1_{\Om} \ast \mu_d  = \1_{\Om} \ast \mu$
a.e., which implies that $\mu = \mu_d$. This shows that
$\mu$, and hence also $\nu$,
must be  a pure point measure.
\end{proof}

\subsection{Weak tiling of an interval by intervals}

We have thus shown that any
 weak tiling measure $\nu$ must be
purely atomic. However, note that
a pure point measure
may have finite accumulation points of atoms.
Hence, we are not guaranteed that $\supp(\nu)$
is a locally finite set, and so we cannot proceed
as in the proof of \thmref{thmD5.4}.

To address this additional difficulty we now
prove two results, which may also be of independent interest,  about
weak tiling of finite intervals and half-lines
by infinitely many intervals.

\begin{thm}
\label{thmG3}
Let $I \sbt \R$ be a bounded open interval,
and suppose that
\begin{equation}
\label{eq:C13.1}
\1_I  = \sum_{j} w_j \1_{I_j} \quad \text{a.e.}
\end{equation}
(a finite or infinite sum) where
$I_j \sbt \R$ are nonempty open intervals,
and $w_j > 0$.
Assume that the lengths of the intervals
$I_j$ belong to some finite set
of positive real numbers $L = \{l_1, \dots, l_n\}$. Then,
\begin{enumerate-roman}
\item \label{it:sfi}
there exist nonnegative integers $s_1, \dots, s_n$
such that $|I| = \sum_{i=1}^{n} s_i l_i$;
\item \label{it:sfiii}
if   $I = (a,b)$, then each one of
the endpoints of any of the intervals $I_j$
is representable in the form
$a + \sum_{i=1}^{n} p_i l_i$
where $p_i$ are some nonnegative integers (depending on $I_j$);
\item \label{it:sfii}
there are only finitely many distinct intervals $I_j$
(so that the sum \eqref{eq:C13.1} is essentially
a finite sum).
\end{enumerate-roman}
\end{thm}

\begin{proof}
First, note that
$I_j \sbt I$ for all $j$. If we set $l := \min \{l_1, \dots, l_n\}$,
 then
\begin{equation}
\label{eq:C11.1}
|I| = \int \1_I =
 \sum_{j} w_j \int \1_{I_j} \ge
 l \sum_{j} w_j,
\end{equation}
which implies (since $l>0$) that $\sum_j w_j < + \infty$.

Second, if some of the intervals $I_j$ coincide, then we may
replace them with a single interval whose weight is the sum
of the weights of the coinciding intervals.
This allows us to assume, with no loss of generality,
that the intervals $I_j$ are distinct.

We may also suppose with no loss of generality  that $I = (0,1)$.
 Let $\Theta$ be the finite subset of $[0,1]$ consisting
of all the real numbers in $[0,1]$
which are representable in the form
$\sum_{i=1}^{n} p_i l_i$
where $p_i$ are nonnegative integers.

We  shall now construct by induction finite sets
$A_k$ and points $x_k \in \Theta$, such that
the following conditions are satisfied:
\begin{enumerate-alph}
\item \label{it:pki}
The intervals $\{ I_j \}$, $j \in A_k$, have their endpoints in $\Theta$;
\item \label{it:pkii}
$\sum_{j \in A_k} w_j \1_{I_j} = 1$ a.e.\ in $[0, x_k]$;
\item \label{it:pkiii}
$\sum_{j \in A_k} w_j \1_{I_j} < 1$ a.e.\ in $[x_k, 1]$.
\end{enumerate-alph}

We begin by taking
$A_0$ to be the empty set and   $x_0 = 0$, so that
\ref{it:pki}, \ref{it:pkii}, \ref{it:pkiii}
hold for $k=0$.
Now suppose that  we have already  constructed
the finite sets
 $A_0, \dots, A_{k}$ and the points
$x_0, \dots, x_{k} \in \Theta$, such that
\begin{equation}
\label{eq:C11.4}
0 = x_0 < x_1 < \dots < x_{k} < 1.
\end{equation}
We will construct a finite set
$A_{k+1}$, and a point $x_{k+1} \in \Theta$
with
$x_{k} < x_{k+1} \le 1$, such that
the conditions
\ref{it:pki}, \ref{it:pkii}, \ref{it:pkiii}
are satisfied with $k$ replaced by $k+1$.

We observe that
$\sum_{j \in A_k} w_j \1_{I_j}$ is
a piecewise constant function, with  only
a finite number of discontinuity points.
It thus follows from \ref{it:pkiii}
that there is $0 \le \lam_k < 1$ such that
$\sum_{j \in A_k} w_j \1_{I_j} = \lam_k$ a.e.\ in some
interval $[x_k, y_k]$, where $x_k < y_k \le 1$.
This and \ref{it:pkii} imply that the sum of the remainder terms
$\sum_{j \notin A_k} w_j \1_{I_j}$
vanishes a.e.\ in $[0, x_k]$, and is equal to
$1 - \lam_k$ a.e.\ in $[x_k, y_k]$.

By taking $y_k$ smaller if needed, we may
 assume that $x_k < y_k \le \min\{1,  x_k + l\}$.

Let $C_k$ be the set of all $j$
such that the left endpoint of $I_j$ is $x_k$.
Note that $C_k$ is a finite set with at
most $n$ elements, since the intervals $I_j$
are distinct and their lengths belong to
$L = \{l_1, \dots, l_n\}$. We claim that if
we set $B_k := C_k \setminus A_k$ then
$\sum_{j \in B_k} w_j = 1 - \lam_k$.

First, it is obvious that
$\sum_{j \in B_k} w_j$ cannot exceed $1 - \lam_k$.
Indeed, otherwise the sum
$\sum_{j \in B_k} w_j \1_{I_j}$ would
exceed $1-\lam_k$
in the interval $(x_k, x_k + l)$,
which in turn implies that the sum
$\sum_{j \in A_k \cup B_k} w_j \1_{I_j}$ must
exceed $\lam_k + (1 - \lam_k) = 1$ a.e.\ in the
 interval $(x_k, y_k)$.
But this is not possible since
$\sum_{j} w_j \1_{I_j} = \1_I$ a.e.

 Hence, to prove that
$\sum_{j \in B_k} w_j$ is equal to $1 - \lam_k$,
it remains to show that this sum
cannot be less than $1-\lam_k$.
Suppose to the contrary that
$\sum_{j \in B_k} w_j  = 1 - \lam_k - \eta$
for some $\eta > 0$.
Let $D_k(\eps)$ be the set of all $j$
such that the left endpoint of $I_j$ belongs to
$(x_k, x_k + \eps)$.
We choose and fix $\eps > 0$ small enough so that
$\sum_{j \in D_k(\eps)} w_j < \eta$.

Now recall that the sum
$\sum_{j \notin A_k} w_j \1_{I_j}$
vanishes a.e.\ in $[0, x_k]$.
Hence, if for some $j \notin A_k$
the interval $I_j$ intersects $[x_k, x_k + \eps)$,
then the left endpoint of $I_j$ must be in
 $[x_k, x_k + \eps)$,
 and therefore
either $j \in B_k$ or $j \in D_k(\eps) \setminus A_k$.
Assuming that $\eps$ was chosen
 small enough so that we also have
  $x_k + \eps \le y_k$, this implies that
\begin{equation}
\label{eq:C8.3.1}
1 - \lam_k = \sum_{j \notin A_k} w_j \1_{I_j}
= \sum_{j \in B_k} w_j \1_{I_j}
+  \sum_{j \in D_k(\eps) \setminus A_k} w_j \1_{I_j}
\quad \text{a.e.\ in $[x_k, x_k + \eps)$.}
\end{equation}
The first sum on the right hand side does not exceed
$\sum_{j \in B_k} w_j  = 1 - \lam_k - \eta$,
while the second sum does not exceed
$\sum_{j \in D_k(\eps)} w_j < \eta$. Hence
the right hand side cannot be equal to $1 - \lam_k$
on any set of positive measure.
We thus arrive at a contradiction, which shows that
we must have
$\sum_{j \in B_k} w_j = 1 - \lam_k$.

We now set  $A_{k+1} := A_k \cup B_k$,
then $A_{k+1}$ is a finite set. We first observe that
the intervals $\{ I_j \}$, $j \in A_{k+1}$,
have their endpoints in $\Theta$. Indeed,
if $j \in A_{k}$ then this holds by the inductive
hypothesis, while if $j \in B_k$ then the
left endpoint of $I_j$ is equal to $x_k$ which is
 in $\Theta$, while the right  endpoint
lies in the set $x_k + \{l_1, \dots, l_n\}$
which again consists of elements of $\Theta$.

Next, we consider the sum
$\sum_{j \in A_{k+1}} w_j \1_{I_j}$
and observe that it is again
a piecewise constant function,
all of whose discontinuity points lie in $\Theta$.
Moreover, we have
\begin{equation}
\label{eq:C8.3.2}
\sum_{j \in A_{k+1}} w_j \1_{I_j}
= \sum_{j \in A_k} w_j \1_{I_j}
= 1
\quad \text{a.e.\ in $[0, x_{k}]$,}
\end{equation}
and
\begin{equation}
\label{eq:C8.3.3}
\sum_{j \in A_{k+1}} w_j \1_{I_j}
= \sum_{j \in A_k} w_j \1_{I_j}
+  \sum_{j \in B_k} w_j \1_{I_j}
= \lam_k + (1 - \lam_k) = 1
\quad \text{a.e.\ in $[x_k, y_k]$.}
\end{equation}
It follows that there exists a point $x_{k+1} \in \Theta$
satisfying  $x_k < y_k \le x_{k+1} \le 1$ and such that
$\sum_{j \in A_{k+1}} w_j \1_{I_j} = 1$
a.e.\ in $[0, x_{k+1}]$, while
$\sum_{j \in A_{k+1}} w_j \1_{I_j} < 1$ a.e.\ in $[x_{k+1}, 1]$.
We conclude that the conditions
\ref{it:pki}, \ref{it:pkii}, \ref{it:pkiii}
are satisfied with $k$ replaced by $k+1$.

Our inductive process may continue as long as $x_k < 1$.
Since $\Theta$ is a finite set, it follows that there
must exist $N$ such that
$x_N = 1$. Hence $|I| = x_N $ is an element of
the set $ \Theta$, which yields condition \ref{it:sfi}.
Moreover, condition
 \ref{it:pkii} with $k=N$ implies that
$ \sum_{j \in A_N} w_j \1_{I_j} = \1_I$ a.e.,
hence the intervals $\{I_j\}$, $j \in A_N$, in fact
constitute all the intervals in the system, and so
condition \ref{it:sfii} of the theorem is established. Finally,
by condition \ref{it:pki} with $k=N$, each one of
the endpoints of any of the intervals $I_j$ must
lie in  $\Theta$, which shows that also
condition \ref{it:sfiii} holds,
and completes the proof.
\end{proof}

The next result is a version of
\thmref{thmG3} for semi-infinite intervals.

\begin{thm}
\label{thmG6}
Let $I  = (a, + \infty) \sbt \R$ be an open half-line,
and suppose that
\begin{equation}
\label{eq:C18.1}
\1_I  = \sum_{j=1}^{\infty} w_j \1_{I_j} \quad \text{a.e.}
\end{equation}
where $I_j  \sbt \R$ are bounded open intervals,
and $w_j > 0$.
Assume that the lengths of the intervals
$I_j$ belong to some finite set
of positive real numbers $L = \{l_1, \dots, l_n\}$. Then,
\begin{enumerate-roman}
\item \label{it:sli}
each one of
the endpoints of any of the intervals $I_j$
is representable in the form
$a + \sum_{i=1}^{n} p_i l_i$
where $p_i$ are nonnegative integers;
\item \label{it:slii}
any bounded subset of $\R$ contains
only finitely many distinct intervals $I_j$.
\end{enumerate-roman}
\end{thm}

\begin{proof}
This can be proved in a similar way as \thmref{thmG3}.
We now proceed with the details.
Again we denote
$l := \min \{l_1, \dots, l_n\}$.
First, we note that if $x \in (a, +\infty)$ then
\begin{equation}
\label{eq:C19.1}
x  - a = \int \1_{(a, x )} \ge
 \sum_{I_j \sbt (a,x)} w_j \int \1_{I_j} \ge
 l \sum_{I_j \sbt (a,x)} w_j,
\end{equation}
hence
the sum of the weights $w_j$
 over all the intervals $I_j$ contained in any fixed interval
 $(a,x)$,  must be finite.
Second, again by replacing coinciding intervals
with a single interval, we may assume
 with no loss of generality
 that the intervals $I_j$ are distinct.

We may also suppose that $a = 0$, and therefore $I = (0,+\infty)$.

 Let $\Theta$ be the set
of all real numbers which are representable in the form
$\sum_{i=1}^{n} p_i l_i$
where $p_i$ are nonnegative integers.
In other words, $\Theta$ is the semigroup generated by the numbers
$l_1, \dots, l_n$. We note that $\Theta$ is a locally finite
subset of $[0, + \infty)$.

Let us now perform the inductive construction from the
proof of \thmref{thmG3}. It yields finite sets
$A_k$ and points $x_k \in \Theta$ with
\begin{equation}
\label{eq:C14.4}
0 = x_0 < x_1 < x_2 < \dots,
\end{equation}
 such that the conditions
\ref{it:pki}, \ref{it:pkii}, \ref{it:pkiii}
are satisfied for each $k$.
The fact that $\Theta$ is a locally finite set,
implies that we must have $x_k \to + \infty$.

We now observe that the intervals
$\{ I_j \}$, $j \in \bigcup_{k=0}^{\infty} A_k$,
must in fact constitute all the intervals in the system.
Indeed, given any $j_0$ we can find $k$ such
that $I_{j_0} \sbt (0, x_k)$. Due to condition
 \ref{it:pkii}, we have
\begin{equation}
\label{eq:C14.7}
1 = \sum_{j \in A_k} w_j \1_{I_j} \le
\sum_{j} w_j \1_{I_j}
= 1 \quad \text{a.e.\ in $[0, x_k]$,}
\end{equation}
which implies that the system $\{I_j\}$, $j \in A_k$,
contains all the intervals $I_j$ that intersect $[0, x_k]$.
In particular, $j_0 \in A_k$
which establishes our claim.
As a consequence, it follows from
 condition \ref{it:pki} that each one of
the endpoints of any of the intervals $I_j$ must
lie in the locally
finite set $\Theta$. This implies both
\ref{it:sli} and \ref{it:slii} and concludes the proof.
\end{proof}

\begin{remark}
An analogous result is true also for
 an open half-line of the form
$I  = (- \infty, a)$, which can be proved in a similar way.
\end{remark}

\subsection{Weak tiling by a finite union of intervals}
We now apply the previous results to the case where
$\Om \sbt \R$ is
 a finite union of intervals that weakly tiles
its complement.

\begin{thm}
\label{thmG2}
Let  $\Om = \bigcup_{i=1}^{n} (a_i, b_i)$,
$a_1 < b_1 < a_2 < b_2 < \dots < a_n < b_n$,
and assume that $\Om$ weakly tiles its complement.
 Then each gap length $a_{k+1} - b_k$
is representable in the form
$ \sum_{i=1}^{n} p_i (b_i - a_i)$
where $p_i$ are nonnegative integers.
\end{thm}

In particular, this result applies if $\Om$
tiles the space (properly) by translations, or if $\Om$ is a spectral set.
In the  latter case, the result recovers
\thmref{thmD5.4} above.

\begin{proof}
Let $\nu$ be a weak tiling measure for $\Om$, then
$\1_\Om \ast \nu = \1_{\Om^\complement}$
a.e. By \thmref{thmG1}, $\nu$ is
 a pure point measure, hence
$\1_{\Om^\cm} = \sum_t \nu(t) \1_{\Om + t}$ a.e., where
$t$ goes through the atoms of $\nu$.
In particular, this implies that if $t$ is an atom of $\nu$, then
$\Om + t \sbt \Om^\cm$.

If we now fix one of the gaps $I = (b_k, a_{k+1}) $, then
$\1_{I} = \sum_t \nu(t) \1_{I \cap (\Om + t)}$ a.e.
If $t$ is an atom of $\nu$, then
the set $I \cap (\Om + t)$
consists of finitely many disjoint open intervals
whose lengths belong to the
finite set $L = \{b_1-a_1, \dots, b_n-a_n\}$.
An application of \thmref{thmG3} thus
allows us to conclude
that the gap length $ a_{k+1} - b_k $ is
representable in the form
$ \sum_{i=1}^{n} p_i (b_i - a_i)$,
where $p_i$ are nonnegative integers.
\end{proof}

Next, we derive some consequences
on the support of the weak tiling measure.

\begin{thm}
\label{thmG7}
Let  $\Om = \bigcup_{i=1}^{n} (a_i, b_i)$,
$a_1 < b_1 < a_2 < b_2 < \dots < a_n < b_n$,
and assume that $\nu$ is a weak tiling measure for $\Om$.
Then the support of $\nu$ is contained in  the locally finite set
consisting of all the nonzero
real numbers which are representable in the form
\begin{equation}
\label{eq:D3.3.8}
\pm \sum_{i=1}^{n} p_i (b_i-a_i)
\end{equation}
where $p_i$
are nonnegative integers.
\end{thm}

\begin{proof}
Let  $\Theta$ denote the set
of all real numbers  which are representable in the form
$\sum_{i=1}^{n} p_i (b_i - a_i)$
where $p_i$ are nonnegative integers, that is,
$\Theta$ is the
semigroup generated by the numbers
$b_1-a_1, \dots, b_n-a_n$.
Then $\Theta$ is a locally finite
subset of $[0, + \infty)$.

By \thmref{thmG1}, $\nu$ is
a pure point measure, hence
$\1_{\Om^\cm} = \sum_t \nu(t) \1_{\Om + t}$ a.e., where
$t$ goes through the atoms of $\nu$.
In particular, for any atom $t$ we have
$\Om + t \sbt \Om^\cm$.

Now consider the half-line
$I = (b_n, + \infty)$ which is contained in $\Om^\cm$, then we have
$\1_{I} = \sum_t \nu(t) \1_{ I \cap (\Om + t) }$ a.e.
We note that in this sum  it  suffices that $t$ only runs through the
 atoms of $\nu$ that are contained in $(0, + \infty)$,
because otherwise $\Om + t $ does not intersect $I$.
For any such $t$, the set $ I \cap (\Om + t) $ is a
union of several of the intervals
\begin{equation}
\label{eq:D3.3}
(a_1 + t, b_1 +t), \quad
(a_2 + t, b_2 +t), \quad
\dots, \quad
(a_n + t, b_n +t),
\end{equation}
whose lengths belong to the
finite set $L = \{b_1-a_1, \dots, b_n-a_n\}$.
In particular, the last interval
$(a_n + t, b_n +t)$ must be one of the components
of $I \cap  (\Om + t)$.
An application of \thmref{thmG6} then
yields that each one of
the endpoints of any of the intervals
that constitute the set $ I \cap (\Om + t) $,
is an element of  $b_n + \Theta$.
In particular, this is the case for the right
endpoint of the interval $(a_n + t, b_n +t)$,
which means that $t \in \Theta$.

We have thus shown that $\supp(\nu) \cap (0, + \infty)$
must be a subset of $ \Theta$. In a similar way, it
follows that
$\supp(\nu) \cap (- \infty, 0)$ is contained in $- \Theta$.
Since the origin obviously cannot be an atom of $\nu$,
due to the weak tiling assumption, we conclude that
\begin{equation}
\label{eq:D3.8.3}
\supp(\nu) \sbt ( \Theta \cup (-\Theta)) \setminus \{0\}.
\end{equation}
This means that the assertion of the theorem holds,
which completes the proof.
\end{proof}

\subsection{Some applications}

We now discuss several applications.
First, suppose that $\Om \sbt \R$
is a finite union of intervals with commensurable lengths. 
By rescaling, we may assume that the lengths are integers.
Theorems \ref{thmG2} and \ref{thmG7} then yield the following:

\begin{thm}
\label{thmG8}
Let  $\Om \sbt \R$ be a finite union of disjoint intervals with integer lengths.
If $\Om$ weakly tiles its complement, then
\begin{enumerate-roman}
\item \label{wtil:i}
all the gaps between the intervals must also have integer lengths; 
\item \label{wtil:ii}
any weak tiling measure $\nu$ for $\Om$ satisfies
 $\supp(\nu) \sbt \Z \setminus \{0\}$.
\end{enumerate-roman}
\end{thm}

Next, assume that $\Om$ is just a single interval.
In this case, we have:

\begin{thm}
\label{thmG9}
If $\Om  = (a,b)$ is a single interval, then any
weak tiling of the complement is necessarily a proper tiling.
\end{thm}

\begin{proof}
We may assume that $\Om  = (0,1)$ is the unit interval. 
If $\nu$ is a weak tiling measure for $\Om$,  then 
$\supp(\nu) \sbt \Z \setminus \{0\}$
by \thmref{thmG7}. Hence
$\1_{\Om^\cm}  = \sum_{n \neq 0} \nu(n) \1_{\Om + n}$.
But the translates
$\Om + n$ are pairwise disjoint intervals, so this is possible
only if $\nu(n) = 1$ for every $n \in \Z \setminus \{0\}$.
This shows that
$\nu = \sum_{n \in \Z \setminus \{ 0 \}} \delta_{n}$,
 so we have a proper tiling.
\end{proof}

Lastly, consider the case where $\Om$ is a union of two intervals.
The following observation was communicated 
to us by the anonymous reviewer.

\begin{thm}
\label{thm:alt}
Let $\Om \sbt \R$ be the union of two disjoint
intervals, and 
assume that $\Om$ weakly tiles its complement. 
If the two intervals have different lengths, then
the weak tiling is necessarily a proper tiling.
\end{thm}

\begin{proof}
After applying a
reflection and translation if necessary, we may assume that 
$\Om = (0,h) \cup (a,b)$ where $0 < h < a < b$, and 
the length $l = b-a$ is greater than $h$.
Let $\nu$ be a weak tiling measure
for $\Om$, then $\nu$ is a pure point
measure whose support is a locally finite set
(due to \thmref{thmG7}). We make the following observation:

\emph{Claim}.
Suppose that $\mu$ is a positive measure on $\R$
with a locally finite support,   such that
$\1_\Om \ast \mu = 1$ a.e.
If $\mu$ has a unit atom at some point $x$,
then also $x + h - a$ and $x + b$ are 
unit atoms of $\mu$.

We now prove the claim. By applying a translation we may assume that $x=0$.

Let us look at the gap $(h,a)$ between the two intervals
$(0,h)$ and $(a,b)$ composing $\Om$.
We observe that a small right neighborhood $(h, h+ \eps)$
of the point $h$ can be covered in the weak tiling
by only two translated copies $\Om + h$ and $\Om + h - a$.
However the 
translated copy $\Om + h$ overlaps with $\Om$,
because the length $l = b-a$ of the interval 
$(a,b)$ is greater than $h$. This implies
that the point $h$ cannot be an atom of $\mu$, 
and hence the point $h-a$ must be a unit atom of $\mu$.

Next, we look at a small right neighborhood $(b, b + \eps)$
of the point $b$. It can again be covered
 in the weak tiling by  only two
translated copies $\Om + b$ and $\Om + b - a$. 
However the 
translated copy $\Om + b-a$ overlaps with $\Om+h-a$.
Since we have already shown that
the point $h-a$ is a unit atom of $\mu$,
this implies
that the point $b-a$ cannot be an atom of $\mu$.
In turn, it follows that the point $b$ must be a unit atom of $\mu$.

We have thus proved the claim.  If we now iteratively 
apply the claim to the measure $\mu := \delta_0 + \nu$,
it implies that in the weak tiling,
whenever there is a short interval of weight one,
then it must be followed by a long interval of weight one;
and similarly, whenever there is a long interval of weight one,
it must be followed by a short interval of weight one.
This is easily seen to force the weak tiling to actually be a proper tiling, 
with the intervals following each other in an alternating manner.
\end{proof}

As a consequence of the last result, we obtain the following:

\begin{thm}
\label{thm:two-intervals}
Let $\Om \sbt \R$ be the union of two disjoint
intervals. If $\Om$ weakly tiles its complement,
then $\Om$ can also tile properly.
\end{thm}

Indeed, if the two intervals have different lengths,
this follows from \thmref{thm:alt}. 
If the   intervals have the same length, then 
by  \thmref{thmG2},
the length of the gap between the two intervals 
must be an integral multiple of the common
length of the intervals. But in this case it
is easy to see that $\Om$ admits a proper tiling.

Since Fuglede's conjecture is known to 
hold if $\Om \sbt \R$ is a  union of two intervals, see \cite{Lab01},
we obtain that in this case, 
spectrality, tiling, and weak tiling of the complement
are equivalent properties.


\section{Open problems}
\label{sec4}

We conclude the paper with a few open problems.

\subsection{}
We say that a set $\Lam \sbt \R$ has \emph{bounded density}
if we have
\begin{equation}
\label{eq:D3.8.22}
\sup_{x \in \R} \# (\Lam \cap (x, x+1)) < + \infty.
\end{equation}
Let $\Om \sbt \R$ be a finite union of intervals,
and let $\nu$ be a weak tiling measure for $\Om$.
Is it true that
$\supp(\nu)$ must be a set of bounded density?

\subsection{}
Let $\Om \sbt \R$ be a finite union of three or more intervals.
If $\Om$ weakly tiles its complement,
must it also tile properly? A positive answer
would imply that the ``spectral implies tile''
direction of Fuglede's conjecture is true for any
finite union of intervals.

\subsection{}
Let $\Om \sbt \R$ be a finite union of intervals,
and let $\nu$ be a weak tiling measure for $\Om$.
Is it true that $\nu$ must be a convex linear combination
of proper tilings? If it is the case, then this implies a positive
answer also to the previous question.

A similar question for convex polytopes in $\R^d$ was posed
in \cite[Section 7.3]{KLM23}.

\subsection*{Acknowledgment}
We thank the anonymous reviewer 
for communicating  to us the observation in \thmref{thm:alt},
as well as
for their insightful comments and suggestions 
to improve the presentation of the manuscript.



\begin{thebibliography}{9999999}

\bibitem[BM11]{BM11}
D. Bose, S. Madan,
Spectrum is periodic for $n$-intervals.
J. Funct. Anal. \textbf{260} (2011), no. 1, 308--325.

\bibitem[DDF25]{DDF25}
B. Ducasse, D. E. Dutkay, C. Fernandez,
Spectral properties of unions of intervals and groups of local translations.
Preprint, \texttt{arXiv:2506.18625}.

\bibitem[Fug74]{Fug74}
B. Fuglede, Commuting self-adjoint partial differential operators and a group
theoretic problem. J. Funct. Anal. \textbf{16} (1974), 101--121.

\bibitem[GL17]{GL17}
R. Greenfeld, N. Lev,
Fuglede's spectral set conjecture for convex polytopes.
Anal. PDE \textbf{10} (2017), no. 6, 1497--1538.

\bibitem[Gru07]{Gru07}
P. M. Gruber, Convex and discrete geometry. Springer, 2007.

\bibitem[Hel10]{Hel10}
H. Helson,
Harmonic analysis. Second edition.
Hindustan Book Agency, New Delhi, 2010.

\bibitem[IK13]{IK13}
A. Iosevich, M. N. Kolountzakis,
Periodicity of the spectrum in dimension one.
Anal. PDE \textbf{6} (2013), no. 4, 819--827.

\bibitem[Kol00]{Kol00}
M. N. Kolountzakis,
Non-symmetric convex domains have no basis of
exponentials. Illinois J. Math. \textbf{44} (2000), no. 3, 542--550.

\bibitem[Kol12]{Kol12}
M. Kolountzakis,
Periodicity of the spectrum of a finite union of intervals.
J. Fourier Anal. Appl. \textbf{18} (2012), no. 1, 21--26.

\bibitem[Kol24]{Kol24}
M. N. Kolountzakis,
\emph{Orthogonal Fourier Analysis on domains.}
Expo. Math. (2024), 125629 (in press).

\bibitem[KL16]{KL16}
M. N. Kolountzakis, N. Lev,
On non-periodic tilings of the real line by a function.
Int. Math. Res. Not. IMRN 2016, no. 15, 4588--4601.

\bibitem[KLM23]{KLM23}
M. Kolountzakis, N. Lev, M. Matolcsi,
Spectral sets and weak tiling.
Sampl. Theory Signal Process. Data Anal. 21 (2023), no. 2, Paper No.\ 31, 21 pp.

\bibitem[KP02]{KP02}
M. Kolountzakis, M. Papadimitrakis, A class of non-convex polytopes
that admit no orthonormal basis of exponentials.
Illinois J. Math. \textbf{46} (2002), no. 4, 1227--1232.

\bibitem[{\L}ab01]{Lab01}
I. {\L}aba,
Fuglede's conjecture for a union of two intervals.
Proc. Amer. Math. Soc. \textbf{129} (2001), no. 10, 2965--2972.

\bibitem[LL21]{LL21}
N. Lev, B. Liu,
Spectrality of polytopes and equidecomposability by translations.
Int. Math. Res. Not. IMRN 2021, no. 18, 13867--13891.

\bibitem[LM22]{LM22}
N. Lev, M. Matolcsi,
The Fuglede conjecture for convex domains is true in all dimensions.
Acta Math. \textbf{228} (2022), no.2, 385--420.

\bibitem[McM80]{McM80}
P. McMullen, Convex bodies which tile space by translation.
Mathematika \textbf{27} (1980), no. 1, 113--121.

\bibitem[McM81]{McM81}
P. McMullen,
Acknowledgement of priority: ``Convex bodies which tile space by translation''
Mathematika \textbf{28} (1981), no. 2, 191.

\bibitem[Tao04]{Tao04}
T. Tao, 
Fuglede's conjecture is false in $5$ and higher dimensions. 
Math. Res. Lett. \textbf{11} (2004), no. 2--3, 251--258.

\bibitem[Ven54]{Ven54}
B. Venkov, On a class of Euclidean polyhedra (Russian). Vestnik Leningrad.
Univ. Ser. Mat. Fiz. Him. \textbf{9} (1954), no. 2, 11--31.


\end{thebibliography}
\end{document}